\documentclass[11pt]{article}
\usepackage{mathrsfs}
\usepackage{amssymb}
\usepackage{amsmath}
\usepackage{amsthm}
\usepackage{amscd}
\usepackage{epstopdf}
\usepackage{enumerate}
\usepackage{hyperref}
\usepackage{listings}

\newtheorem{theorem}{Theorem}[section]
\newtheorem{conjecture}{Conjecture}

\newtheorem{lemma}[theorem]{Lemma}
\newtheorem{proposition}[theorem]{Proposition}

\newtheorem{definition}[theorem]{Definition}

\theoremstyle{definition}

\newcommand{\D}{\mathbb{D}}

\newcommand{\Q}{\mathbb{Q}}
\newcommand{\Z}{\mathbb{Z}}
\newcommand{\R}{\mathbb{R}}
\newcommand{\C}{\mathbb{C}}
\newcommand{\F}{\mathbb{F}}
\newcommand{\Fq}{\mathbb{F}_q}
\newcommand{\Fpbar}{\overline{\mathbb{F}}_p}
\newcommand{\A}{\mathscr{A}}
\newcommand{\G}{\mathscr{G}}
\newcommand{\Cl}{\mathcal{C}\ell}
\newcommand{\Ok}{\mathcal{O}}

\DeclareMathOperator{\Hom}{Hom}
\DeclareMathOperator{\End}{End}
\DeclareMathOperator{\Trace}{Trace}
\DeclareMathOperator{\Norm}{Norm}

\title{Almost ordinary abelian varieties over finite fields}
\date{\today}
\author{Abhishek Oswal and Ananth N. Shankar }

\begin{document}
\maketitle
\begin{abstract}
We provide a characterization of almost ordinary abelian varieties over finite fields, and use this characterization to provide lower bounds  for the sizes of some almost ordinary isogeny classes. 
\end{abstract}

\section{Introduction}
In his work \cite{Deligne}, Deligne provides a classification of ordinary abelian varieties over finite fields. The key ingredient required for this classification is the existence of the so-called ``canonical lift'' to characteristic zero of an ordinary abelian variety. The content of this paper is to provide a similar classification of certain abelian varieties over finite fields with odd characteristic $p$, which are \emph{almost ordinary}, and use this classification to estimate the size of certain isogeny classes. We note that Centeleghe and Stix in \cite{CJ} also have a characterization of abelian varieties which aren't necessarily ordinary (therefore generalizing \cite{Deligne}). However, their characterization is for abelian varieties over the prime field $\F_p$, and their methods do not make use of lifts to characteristic zero. 

We define a ``simple almost ordinary abelian variety'' over a finite field to be a $g$-dimensional abelian variety over a finite field which is geometrically simple and has $p$-rank equal to $g-1$, where $g \geq 2$. Honda-Tate theory implies that the endomorphism algebra of such an abelian variety equals a CM field $K$. The slope 1/2 part of $A$ corresponds to $K_{ss}$, a quadratic extension of $\Q_p$, which is contained in $K \otimes \Q_p$. We call the almost ordinary abelian variety \emph{ramified} $K_{ss}$ is a ramified extension of $\Q_p$ and \emph{inert} if otherwise. 

The main result of this paper is 

\begin{theorem}\label{class}
 Let $\mathcal{C}_h$ denote the category of simple almost ordinary abelian varieties corresponding to the Frobenius-polynomial $h$. 
 Let $\mathcal{L}_h$ denote the category of almost ordinary Deligne modules \footnote{See \S 3 and Definition \ref{almostdefn} for the precise definitions}, also corresponding to the Frobenius polynomial $h$. 
 \begin{enumerate}
     \item Suppose that $\mathcal{C}_h$ is ramified. There exist two functors $\mathfrak{T}_1,\mathfrak{T}_2$ from $\mathcal{C}_h$ to $\mathcal{L}_h$, both of which induce an equivalence of categories. 
     
     \item Suppose that $\mathcal{C}_h$ is inert. There are two full subcategories $\mathcal{C}_{1,h}$ and $\mathcal{C}_{2,h}$ of $\mathcal{C}_h$, and functors $\mathfrak{T}_i$ from $\mathcal{C}_{i,h}$ to $\mathcal{L}_h$, both of which induce an equivalence of categories. Further, the objects of the category $\mathcal{C}_h$ is a disjoint union of the objects of $\mathcal{C}_{1,h}$ and $\mathcal{C}_{2,h}$.
     
 \end{enumerate}
 
\end{theorem}
Unlike the case of ordinary abelian varieties, there is no unique canonical functor between $\mathcal{C}_h$ and $\mathcal{L}_h$. This is because there two different (and equally canonical) choices for a CM type on $K = \Q[x]/(h(x))$. Indeed, (after fixing an embedding of the algebraic closure of $\Q_p$ in $\C$) there are $g-1$ complex embeddings of $K$  corresponding to the slope $0$, $g-1$ embeddings corresponding to the slope 1 and 2 embeddings corresponding to the slope $1/2$. The embeddings corresponding to the slopes 0 all must lie in one of the two sets making up the partition induced by a CM type on $K$, and the embeddings corresponding to the slopes 1 must all lie in the other set making up the partition. The two embeddings corresponding to the slope 1/2 can be partitioned into the two sets in exactly two ways, hence the ambiguity. 

%Some text talking about choice of CM type and relating it to the ambiguity / failure to be a functor in the generality that Deligne's is. 

One of the key steps in proving Theorem \ref{class} is to construct an analogue of the ``canonical lift'' for a simple almost ordinary abelian variety $A$. Deuring in \cite{Deuring} proved that given any endomorphism of a supersingular elliptic curve, there exists a lift to characteristic zero of this endomorphism. The $\F_q$-structure on the supersingular part of $A[p^{\infty}]$ isolates a unique rank two subalgebra of $\End(A[p^{\infty}]_{ss}\times_{\F_q}\overline{\F}_q)$, and hence  a choice of lift of $A[p^{\infty}_ss]$ to mixed characterisic (see \S 2 for details). Using Grothendieck-Messing theory, we show that there is a canonical choice of lift in the inert case, and two equally canonical choices of lift in the ramified case. That there are two different choice of lift in the ramified case corresponds to the ambiguity in picking a CM type on $K$. In the inert case, even though there is a canonical choice of lifting, the association of the canonical lift to $A$ is \emph{not} functorial in $A$. It \emph{is} functorial if we restrict to any one subcategory mentioned in the statement of Theorem \ref{class}. This can be interpreted by saying that choosing a CM type on $K$ is equivalent to choosing one of the two subcategories. See Section 3 for more details. 

\subsection*{Application to estimating the size of principally polarized isogeny classes}
Our characterization of simple almost ordinary abelian varieties is robust enough to deal with polarizations and duals (see Section 4). We make the following conjecture: 
\begin{conjecture}
Let $\mathcal{C}_h$ denote an almost ordinary isogeny class defined over $\F_q$. Then, for a positive proportion of integers $n$, the number of principally polarized $\F_{q^n}$-rational abelian varieties in $\mathcal{C}_h$ is $q^{n(\dim \mathcal{A}_g -1)(1/2 + o(1))}$. Further, the right hand side is an upper-bound for all $n$. 
\end{conjecture}
This agrees with Conjecture 3.1 of \cite{AJ}. Indeed, the quantity $\dim \mathcal{A}_g$ in \cite{AJ} is replaced by $\dim \mathcal{A}_g -1$ because Newton Stratum consisting of almost ordinary abelian varieties is codimension 1 in $\mathcal{A}_g$, (as opposed to being of codimension 0, as in the ordinary case). We establish the lower bound for a positive proportion of $n$, for isogeny classes where the associated Frobenius Torus (see \cite[Section 3a]{Chi} for the definition of Frobenius torus) has maximal rank. The precise result is stated as Theorem \ref{lastmain}.

\subsection*{Plan for the rest of the paper}
We construct the canonical lift(s) in Section 2. We prove Theorem \ref{class} in section 3, and address the matter of polarizations in section 4. We apply the results of Sections 3 and 4 to establish lower bounds in Section 5. 
\subsection*{Acknowledgements} 
It is a pleasure to thank Bjorn Poonen, Yunqing Tang and Jacob Tsimerman for helpful discussions. We are also very grateful to Steven Kleiman and Rahul Singh for answering our questions pertaining to Gorenstein rings.

\section{The canonical lift}
In this section, we will construct the ``canonical lift'', characterized by property that every endomorphism lifts. Further, this will be the unique (or the two unique) CM lifts to a slightly ramified extension\footnote{See Definition \ref{slightram} for the definition of slightly ramified} of $W:=W(\overline{\F}_p)$. We will first need some preliminary results about the endomorphism rings of almost ordinary abelian varieties. 

\subsection{Endomorphism rings}

Let $A$ denote an almost ordinary abelian variety over $\Fq$, and let $\G$ denote its $p$-divisible group. Let $R = \End(\A)$, and let $S = \End(\G)$ (note that $S$ is the endomorphism ring of $\G$ over $\Fq$, and not $\Fpbar$). Tate's theorem states that $S = R \otimes_{\Z}\Z_p$. We have the following result:

\begin{proposition}\label{max}
The $\Z_p$-algebra $S$ is of the form $S_{et} \oplus S_{tor} \oplus S_{ss}$ according to the decomposition of $\G = \G_{et} \times \G_{tor} \times \G_{ss}$. Further, if $A$ is geometrically simple, then $R$ is an order inside a CM field of degree $2g$. Consequently, $S_{ss}$ is a 2-dimensional $\Z_p$-algebra. Further, $S_{ss}$ is the maximal order in its field of fractions. 
\end{proposition}

\begin{proof}
As $\F_q$ is a perfect field, every $p$-divisible group is a product of its etale, toric and local-local components, and so the first assertion follows. 

In order to prove the rest it suffices to treat the case where $A$ is simple. Let $\pi_A$ denote the Frobenius endomorphism of $A$. Let $K = Q(\pi_A)$ and $L = R \otimes \Q$. We will prove that $K = L$ in order to show that $S_{ss}$ is a rank 2 $\Z_p$-algebra. Let $h(x)$ and $f(x)$ denote the characterstic polynomial and minimal polynomial of $\pi_A$ respectively. As $A$ is simple and almost ordinary, $h(x)$ is either $f(x)$, or $f(x)^2$ according to whether $R$ is commutative or not. 

We now show that $f(x) = h(x)$. Indeed, let $f(x) = f_{et}(x) \cdot f_{tor}(x) \cdot f_{ss}(x)$ over $\Q_p$ according to the slope decomposition of $A$, and similarly let $h = h_{et} \cdot h_{tor} \cdot h_{ss}$ over $\Q_p$. If $h(x) = f(x)^2$, then $f_ss$ is a degree one polynomial, and so let $\pi$ denote its root. This implies that $h_{ss} = (x - \pi)^2$. On the other hand, the product of the roots of $h_{ss}$ is $q$, and so $\pi = q / \pi$, which implies that $\pi = \pm q^{1/2}$. However, this implies that the minimal polynomial of Frobenius of $A \times_{\F_q}\F_{q^2}$ is reducible over $\Q$, which contradicts the assumption that $A$ is geometrically simple. It follows that $h(x) = f(x)$, and therefore that $R$ and $S_{ss}$ are commutative.

It remains to prove that $S_{ss}$ is the maximal order in its field of fractions. Let $A'$ denote an abelian variety over $\F_q$ isogenous to $A$ such that $\End(A')$ is the maximal order in its quotient algebra and denote by $\G'$ its $p$-divisible group. Let $S'_{ss}$ be the endomorphism ring of $\G'_{ss}$. Let $i: A_0 \rightarrow A$ be an isogeny, and let $j$ denote the associated map from $\G_0$ and $\G$. Clearly, $j$ breaks up as $j_{ord} \times j_{ss}$. It suffices to show that $S'_{ss}$ preserves the kernel of $j_{ss}$. Indeed, $\G'_{ss}$ is a connected one-dimensional group, and hence $\G'_{ss}$ has a unique subgroup of order $p^m$ for every positive integer $m$. Therefore, it follows that $S'_{ss}$ preserves the kernel of $j_{ss}$ and the result follows. 

\end{proof}

We will now use Grothendieck-Messing theory to prove Deuring's lifting theorem. We first define the notion of a ``slightly ramified'' extension, as Grothendieck-Messing theory does not apply when the degree of ramification is large. 
\begin{definition}\label{slightram}
Let $L$ denote a finite extension of $W[1/p]$. We say that $L$ and $\Ok_L$ are slightly ramified if the degree $[L:W[1/p]]$ is at most $p-1$.
\end{definition}

\begin{lemma}\label{ss}
Let $\G_s$ denote a one-dimensional supersingular $p$-divisible group over $\Fpbar.$ Suppose $\Ok \subset \End(\G_s)$ denote an integrally closed rank two $\Z_p$ algebra. Then: 
\begin{enumerate}
    \item If $\Ok$ is unramified, there exists a unique lift of $\G_s$ to $W$ such that the action of $\Ok$ lifts. 
    
    \item If $\Ok$ is ramified, then there exist two lifts of $\G_s$ to $W[\sqrt{p}]$ such that the action of $\Ok$ lifts.
\end{enumerate}
In either of the above cases, the lifts described are the only ones to a slightly ramified extension of $W$ such that the action of $\Ok$ lifts. 

\end{lemma}
\begin{proof}
Let $\D$ denote the Dieudonne module of $\G_s$ -- $\D$ is a free rank 2 $W$-module, equipped with a Frobenius-semilinear endomorphism which we denote by $F$. The module $\D$ has a basis $e_1,e_2$ such that $Fe_1 = e_2, Fe_2 = pe_1$. 

The endomorphisms of $\G_s$ consist of the $W$-linear endomorphisms of $\D$ which $\sigma$-commute with $F$. Every such endomorphism is easily seen to be of the form 

\[
   M_{a,b}=
  \left[ {\begin{array}{cc}
   a & p\sigma(b) \\
   b & \sigma(a) \\
  \end{array} } \right]
\]
where $a,b \in W(\F_{p^2})$. For ease of notation, we will identify the matrix $M_{a,b}$ with the endomorphism of $\G_s$ that it represents. Note that the characteristic polynomial of $M_{a,b}$ is $x^2 - (a + \sigma(a))x + (a\sigma(a) - pb\sigma(b)$. 

As $\Ok$ is a rank-two $\Z_p$ module, it is monogenic as a $\Z_p$-algebra, and so we may assume that $\Ok$ is generated by a single trace-zero endomorphism. Therefore, there exist $a,b$ such that $O = \Z_p[M_{a,b}]$, such that $\sigma(a) = -a$. Further, $\Ok$ is the maximal order in its field of fractions, which is equivalent to $\Ok$ having square-free discriminant. The discriminant of $\Ok$ equals the discriminant of the characteristic polynomial of $M_{a,b}$, which is $-4(a\sigma(a) - pb\sigma(b)$ which is squarefree if and only if at least one among $a,b$ is a $p$-adic unit. Further, $\Ok$ is unramified if and only if $a$ is a $p$-adic unit. 

Let $R$ denote the ring of integers of some slightly ramified extension of $W$. As the extension is slightly ramified, the  maximal ideal of $R$ is closed under divided powers. Therefore, Grothendieck-Messing theory applies, and yields the following statement (see \cite[Section 5, Theorem 1.6]{GM}):
Deformations to $R$ of $\G_s$ such that the action of $M_{a,b}$ also lifts are in bijection with a rank one co-torsion free submodule $Fil \subset \D \otimes_W R$, such that $Fil$ reduces to $e_2$ modulo the maximal ideal of $R$. 

We will now show that there is a unique choice of $Fil$ if $a$ is a $p$-adic unit, and that there are exactly two choices of $Fil$ otherwise. Indeed, the data of $Fil$ is the same as the data of an Eigenvector of $M_{a,b}$ of the form $e_2 + \lambda e_1$, where $\lambda$ is in the maximal idea of $R$. 

Having fixed $a,b$, a vector of the form $e_2 + e_1 \lambda$ is an eigen vector of $M_{a,b}$ if and only if $\lambda$ satisfies the quadratic equation 
\begin{equation}\label{quad}
b \lambda^2 + (\sigma(a) - a) \lambda - p\sigma(b) = 0 
\end{equation}
Note that this already proves for us the statement that there are at most two deformations of $\G_s$ to a slightly ramified extension of $W$, such that the action of $M_{a,b}$ also lifts. We will now treat the following two cases to finish the proof of this lemma:

\subsubsection*{Case 1: $a$ is a $p$-adic unit.}
We must prove that there is a unique $\lambda$ with positive $p$-adic valuation that satisfies \eqref{quad}, and that this solution lies in $W$. Indeed, the newton polygon of \eqref{quad} has a breakpoint, and so equals a product of linear factors, thereby proving that any $\lambda$ is an element of $W$. Further, the product of the two roots has $p$-adic valuation 1, and the sum of the two roots (which equals $\frac{a - \sigma(a)}{b})$ has $p$-adic valuation non-positive. This is because $\sigma(a) = -a$, and $a$ is a unit. Therefore, exactly one of the two solutions to \eqref{quad} has positive $p$-adic valuation, as required. 

\subsubsection*{Case 2: $a$ is not a $p$-adic unit.}
Recall that $b$ has to be a $p$-adic unit in this case. It follows that the discriminant of \eqref{quad} has $p$-adic valuation one, and so must be an irreducible polynomial over $W$, and thus the two roots have the same $p$-adic valuation. As the product of the roots has $p$-adic valuation 1, each of the two roots must have $p$-adic valuation $1/2$, and so must be defined over $W[\sqrt{p}]$. The lemma follows.  

\end{proof}

\subsection{Definition of the canonical lift(s)}

We will now define the canonical lift(s) of our abelian variety. 

\begin{definition}
Let $\tilde{\G}_s$ denote the lift(s) of $\G_s$ constructed in Lemma \ref{ss}. We define the canonical lift(s) $\G_{can}$ of $\G$ to be $\tilde{\G}_s \times \tilde{\G}_{et} \times \tilde{\G}_{tor}$. We define the canonical lift(s) of $A$ to be the abelian variety corresponding to $\G_{can}$ via the Serre-Tate lifting equivalence \footnote{See \cite{Drin} for a proof of the lifting equivalence due to Drinfel'd}.
\end{definition}

\begin{proposition}
The canonical lift has the property that all the endomorphisms of $A$ lift. 
\end{proposition}
\begin{proof}
It suffices to show that all the $\F_q$-endomorphisms of $\G$ lift to $\G_{can}$. Recall that $\End_{\F_q}(\G) = S_{et} \oplus S_{tor} \oplus S_s$. By construction, the action of $S_s$ lifts to $\G_{can}$. It suffices to show that the actions of $S_{et}$ and $S_{tor}$ also lift. This follows, because $\tilde{\G}_{et} \times \tilde{\G}_{tor}$ is the canonical lift of $\G_{et} \times \G_{tor}$, and the canonical lift of an ordinary $p$-divisible group is characterized by the property that every endomorphism lifts. 
\end{proof}

\section{Classification}
\textbf{For this section, fix an odd prime $p$ and $q = p^a$.} 

A remarkable application of the Serre-Tate canonical lift for ordinary abelian varieties was given by Deligne where he provides a classification of ordinary abelian varieties over finite fields in terms of a certain category of $\Z$-modules. More precisely, if $A/\mathbb{F}_q$ is an ordinary abelian variety, let $\tilde{A}/W(\mathbb{F}_q)$ denote its canonical lift. Fixing an embedding $\iota : W(\mathbb{F}_q) \hookrightarrow \mathbb{C},$ Deligne considers the integral homology $T := H_1(\tilde{A}\otimes_\iota \mathbb{C},\mathbb{Z})$. The Frobenius endomorphism of $A$ over $\mathbb{F}_q$ lifts to an endomorphism of $\tilde{A}$ and thus defines an endomorphism $F \in \text{End}_\mathbb{Z}(T).$ Deligne then shows that the association that takes $A/\mathbb{F}_q$ to the pair $(T,F)$ gives an equivalence of categories between ordinary abelian varieties over $\mathbb{F}_q$ of dimension $g$ and the category of pairs $(T,F)$ where $T$ is a free $\mathbb{Z}$-module of rank $2g$ and $F \in \text{End}_\mathbb{Z}(T)$ satisfies the following three conditions: \begin{enumerate}\label{conditions}
	\item\label{Weil} $F\otimes \Q$ acts semisimply on $T\otimes \Q$.
	\item\label{Ver} There exists $V \in \text{End}_\mathbb{Z}(T)$ such that $F \circ V = q = V\circ F.$
	\item\label{Slope} The characteristic polynomial $h(x) \in \Z[x]$ of $F$ is a Weil $q$-polynomial (i.e. all its roots are Weil $q$-integers) such that $h(x)$ has at least $g$ roots (counting multiplicities) in $\overline{\mathbb{Q}}_p$ that are $p$-adic units.
\end{enumerate}  

Our goal in this section is to provide a similar classification for simple almost ordinary abelian varieties over $\mathbb{F}_q$ using our canonical lift(s) defined above. 

Suppose $A /\mathbb{F}_q$ is a \emph{simple} almost ordinary abelian variety of dimension $g$. Fix an embedding $\epsilon : W(\overline{\mathbb{F}_q})[\sqrt{p}] \hookrightarrow \mathbb{C}$. Let $\tilde{A}$ be one of the possibly two canonical lifts of $A$ over a ramified quadratic extension of $W(\mathbb{F}_q)$ and consider $T(A) := H_1(\tilde{A}\otimes_{\epsilon}\mathbb{C},\mathbb{Z})$, a free $\mathbb{Z}$-module of rank $2g$. The Frobenius endomorphism of $A$ over $\mathbb{F}_q$ lifts to an endomorphism of $\tilde{A}$ and thus defines an $F(A) \in \text{End}_\mathbb{Z}(T(A)).$ We associate the pair $(T(A),F(A))$ to $A$. Note that if $\mathcal{G}$ is the $p$-divisible group of $A$ and when $\End(\mathcal{G}_{\text{ss}})$ is ramified over $\Z_p,$ then $A$ has two canonical lifts over $W(\overline{\mathbb{F}}_q)[\sqrt{p}]$ and to each such canonical lift we associate a pair $(T,F).$ As in the ordinary case, \ref{Weil}. and \ref{Ver}. above are satisfied. However, we shall see that \ref{Slope} is replaced by
\begin{enumerate}
	\item[$3^*$]\label{Slope*} The characteristic polynomial $h(x) \in \mathbb{Z}[x]$ of $F$ is a Weil $q$-polynomial which is irreducible over $\Q$, and has $g-1$ roots in $\overline{\mathbb{Q}}_p$ that are $p$-adic units, $g-1$ roots with $q$-adic valuation $1$ and 2 roots with $q$-adic valuation 1/2. Thus, we have a factorisation in $\mathbb{Z}_p[x]$
	\[h(x) = h_0(x)\cdot h_1(x)\cdot h_{1/2}(x)\] where $h_i(x)$ has all its roots in $\overline{\mathbb{Q}}_p$ with $q$-adic valuation $i$. Moreover, note that $h_{1/2}(x)$ must be irreducible over $\mathbb{Q}_p$ and does not have $ \pm \sqrt{q}$ as a root. 
	
	(Indeed, if $h_{1/2}(x)$ has two distinct roots in $\mathbb{Q}_p$ then the supersingular part of the $p$-divisible group of $A$ has endomorphism algebra $\mathbb{Q}_p \times \mathbb{Q}_p$ which is not possible. If $a$ is even, $\pm \sqrt{q}$ cannot be a root since otherwise $A$ would have as an isogeny factor a supersingular elliptic curve with all its endomorphisms defined over $\mathbb{F}_q$. Similarly, if $a$ is odd and if $\pm \sqrt{q}$ is a root of $h(x)$ then this would imply that $(x^2 - q)^2$ divides $h(x)$ contradicting that $A$ is almost ordinary.
	
	And to see that $h$ is irreducible over $\Q$ we note that we're working under the assumption that $A$ is simple. Thus, $h(x)$ is a power of a $\Q$-irreducible polynomial $P$, say $h(x) = P(x)^e$. We may compute $e$ as in \cite[Pg. 527]{Waterhouse}. Factor $P(x) = \prod P_\nu (x)$ into irreducible factors in $\mathbb{Q}_p[x].$ Since $P$ has no real roots, $e$ is the least common denominator of the $i_\nu := \frac{\text{ord}_p P_\nu (0)}{a}.$ Note that $h_{1/2}$ occurs as one of the $P_\nu$, and $\frac{\text{ord}_p (h_{1/2}(0))}{a} = 1.$ The other $P_\nu,$ divide either $h_0$ or $h_1$. If $P_\nu$ divides $h_0$ then all the roots of $P_\nu$ are $p$-adic units and hence $i_\nu = 0,$ and if $P_\nu$ divides $h_1$ then all the roots of $P_\nu$ have $p$-adic valuation $a$. In all cases, $i_\nu$ is an integer and hence $e=1.$)

\end{enumerate}

Note that under the assumptions \ref{Weil}, \ref{Ver} and $3^*$ on the pair $(T(A),F(A))$ we have a decomposition 
\[ T(A) \otimes \mathbb{Z}_p = T_0 \oplus T_1 \oplus T_{1/2}\] where $T_i := \text{Ker}(h_i(F)).$ Thus, $T_{1/2}$ is a rank 1 module over $\mathbb{Z}_p[x]/h_{1/2}(x).$ In fact, it follows from Proposition \ref{max} that 
\begin{enumerate}
    \item[$4^*$] \,$T_{1/2}$ has endomorphisms by the \emph{maximal} order $\mathcal{O}_{ss}$ of $K_{ss}:=\mathbb{Q}_p[x]/\langle h_{1/2}(x)\rangle .$
\end{enumerate}

\begin{definition}\label{almostdefn}
\begin{enumerate}
    \item A pair $(T,F)$ satisfying the four conditions \ref{Weil},\ref{Ver}, 3* and 4* is said to be an almost ordinary Deligne module with Frobenius polynomial $h$.
    \item A morphism of almost ordinary Deligne modules $\phi : (T,F) \rightarrow (T',F')$ is simply a morphism $\phi : T \rightarrow T'$ of $\Z$-modules such that $\phi \circ F = F'\circ \phi$
    \item  An isogeny of almost ordinary Deligne modules is a morphism $\phi : (T,F) \rightarrow (T',F')$ such that $\phi \otimes \Q : T\otimes \Q \rightarrow T'\otimes \Q$ is an isomorphism.
    \item For a polynomial $h(x) \in \Z[x]$ satisfying 3*, we denote by $\mathcal{L}_h$ the category of almost ordinary Deligne modules with Frobenius polynomial $h$. Similarly, we define $\mathcal{C}_h$ as the category of simple almost ordinary abelian varieties over $\F_q$ with Frobenius polynomial $h$.
\end{enumerate}
\end{definition}

\begin{definition}
We say that $\mathcal{L}_h$ or $\mathcal{C}_h$ (or an object of either category) is ramified (resp. inert) when $K_{ss}$ is a ramified (resp. unramified) quadratic extension of $\Q_p.$
\end{definition}

Thus, in case that $A \in \mathcal{C}_h$ is ramified, we obtain two almost ordinary Deligne modules $\mathfrak{T}_1(A), \mathfrak{T}_2(A)$ in $\mathcal{L}_h$ using the two possible canonical lifts of $A$. 
 
\begin{proposition}\label{surjectivity}
Every almost ordinary Deligne module $(T,F) \in \mathcal{L}_h$ arises (up to isomorphism) from a simple almost ordinary abelian variety over $\F_q$ 
\end{proposition}
\begin{proof}
By Honda-Tate theory we may find a simple almost ordinary abelian variety $A$ over $\mathbb{F}_q$ such that the characteristic polynomial of the Frobenius $\pi_A$ (relative to $\mathbb{F}_q$) is $h(x).$ Thus, we see that $(T(A)\otimes \mathbb{Q},\pi_A) \cong (T\otimes \mathbb{Q},F).$

By making the above identification we have that $T \subset (T(A)\otimes \mathbb{Q}, \pi_A)$ is a $\pi_A$-stable lattice of full rank. We aim to find an almost abelian variety $B$ isogenous to $A$, such that $(T(B),\pi_B) \cong (T,\pi_A).$ For this, we may assume that $T(A) \subseteq T.$ Then $T/T(A)$ defines a finite subgroup $H \subseteq \tilde{A}\otimes_\epsilon \mathbb{C},$ stable under the lift to characteristic $0$ of the Frobenius $\pi_A$ of $A$. If the order of $H$ is coprime to $p$ then $H$ defines a subgroup scheme of $\tilde{A}.$ If $\overline{H} \subset A$ denotes the subgroup scheme obtained by reducing modulo $p$ we see that $\tilde{A}/H$ is a canonical lift of $A/\overline{H}$ and moreover, $H_1(\tilde{A}/H \otimes_\epsilon \mathbb{C}, \mathbb{Z}) = T.$  On the other hand, let $H = T/T(A)$ be a $p$-group. Then, corresponding to the decomposition $T \otimes \mathbb{Z}_p = T_0 \oplus T_1 \oplus T_{1/2}$ we have a decomposition of $H \otimes \mathbb{Z}_p = H_0 \oplus H_1 \oplus H_{1/2}.$ Moreover, $H_0$ lifts uniquely an \'etale subgroup $\overline{H}_0 \subset A$, and similarly $H_1$ lifts $\overline{H}_1$ in the toric part of the $p$-divisible group of $A.$ For the group $H_{1/2} = T_{1/2}/T(A)_{1/2},$ we note that since both $T_{1/2}$ and $T(A)_{1/2}$ admit endomorphisms by the maximal order $\mathcal{O}_{ss} \subseteq \mathbb{Q}_p[x]/h_{1/2}(x),$ we must have that $T(A)_{1/2} = \varpi^n T_{1/2}$ where $\varpi \in \mathcal{O}_{ss}$ is a uniformizer. Thus, $H_{1/2}$ lifts the supersingular part of the kernel of the isogeny given by $\varpi^n$ on $A.$ Hence, $H$ lifts a unique subgroup $\overline{H}\subset A,$ such that $\tilde{A}/H$ is a canonical lift of $B := A/\overline{H}.$ Moreover, it is clear that $(T(B),\pi_B) = (T,\pi_A).$

\end{proof}
 
 \subsection{Inert isogeny classes}\label{inertclasses}
 
Let $\alpha \in \End(A)$, where $A$ is some inert simple almost ordinary abelian variety. Let $\alpha_{ss} \in \Ok_{ss}$  denote the restriction of $\alpha$ to the local-local part of $A[p^{\infty}]$. It is easy to see that the order of $\ker(\alpha_{ss})$ equals $\Norm(\alpha_{ss})$, and as $\Ok_{ss}[1/p]$ is an unramified extension of $\Q_p$, it follows that the order has to be an \emph{even} power of $p$. Motivated by this, given an inert isogeny class $\mathcal{C}_h$, we define the following equivalence relation on its set of objects $\text{Ob}(\mathcal{C}_h)$. For $A,B \in \text{Ob}(\mathcal{C}_h)$ we say that $A \sim B$ when some (hence \emph{every}) $\F_q$-isogeny $f : A \rightarrow B$ is such that if $f[p^\infty] : A[p^\infty] \rightarrow B[p^\infty]$ denotes the induced map of $p$-divisible groups then $\ker(f[p^\infty])_{ss}$ has order $p^{r}$ for an \emph{even} integer $r$. The equivalence relation partitions the objects of $\mathcal{C}_h$ into two equivalence classes and we let $\mathcal{C}_{1,h}$ and $\mathcal{C}_{2,h}$ denote the two full subcategories of $\mathcal{C}_h$ having objects of each equivalence class. Thus $\text{Ob}(\mathcal{C}_h) = \text{Ob}(\mathcal{C}_{1,h}) \cup \text{Ob}(\mathcal{C}_{2,h})$. We claim that:
 \begin{proposition}\label{inertclassification}
 Restricted to each equivalence class $\mathcal{C}_{i,h}$ the association $A \mapsto \mathfrak{T}_i(A):=(T(A),F(A))$ is functorial for $A \in \mathcal{C}_{i,h}$ and moreover induces an equivalence of categories $$\mathfrak{T}_{i} : \mathcal{C}_{i,h} \rightarrow \mathcal{L}_h.$$
 \end{proposition}
 
 \begin{proof}
 Note that if $f :A \rightarrow B$ is an $\F_q$-isogeny for $A,B \in \mathcal{C}_{i,h}$, then since $\ker(f[p^\infty])_{ss}$ is of order $p^{2m}$ for $m \in \Z,$ $f[p^\infty]$ lifts uniquely to a morphism of the lifts of the $p$-divisible groups $\widetilde{f[p^\infty]} : \widetilde{A}[p^\infty]\rightarrow \widetilde{B}[p^\infty]$ and hence to a morphism between the canonical lifts $\widetilde{f} : \widetilde{A}\rightarrow \widetilde{B}.$ Thus, we obtain a morphism of integral homologies $\mathfrak{T}_i(A)\rightarrow \mathfrak{T}_i(B) .$ It follows easily that $\mathfrak{T}_i$ is a functor.
 
 The fact that each $\mathfrak{T}_i$ is essentially surjective follows from the proof of Proposition \ref{surjectivity}. We simply note that in the proof we may as well have started with an $A \in \mathcal{C}_{i,h}$ such that $(T(A)\otimes \Q,\pi_A) \cong (T\otimes \Q,F).$ Then it suffices to see that the $B = A/\overline{H}$ obtained at the end of the earlier proof belongs to the same equivalence class as $A$, since the order of $\overline{H}[p^\infty]_{ss}$ is indeed a square when $p$ is inert in $\mathcal{O}_{ss}$.    

It remains to show that $\mathfrak{T}_i$ is fully-faithful, i.e. for $A,B \in \mathcal{C}_{i,h}$ we show that the natural map $\Hom_{\F_q}(A,B) \rightarrow \Hom_{\mathcal{L}_h}(\mathfrak{T}_i(A),\mathfrak{T}_i(B))$ is a bijection. The injectivity of this map is clear. Indeed, different isogenies between $A$ and $B$ lift to different isogenies between $\tilde{A}$ and $\tilde{B}$, and different maps between $\tilde{A}$ and $\tilde{B}$ induce different maps between $T(B)$ and $T(C)$.

To show that an isogeny $g : T(A) \rightarrow T(B)$ is in the image, it suffices to show that $n\cdot g$ is in the image, for some integer $n$. Indeed, if $u : A \rightarrow B$ is an $\F_q$-isogeny such that $\mathfrak{T}_i(u)$ is divisible by $n$, then this means that $\tilde{u}\otimes_\epsilon \C$ is divisible by $n$, and hence $\tilde{u}$ is divisible by $n$ over the generic point of $W(\F_q).$ Since the kernel of $n$ is flat over $W$ we get that $\tilde{u}$ and hence $u$ is divisible by $n$. 

Let $\phi :A \rightarrow B$ be an $\F_q$-isogeny. Note that $\mathfrak{T}_i(\phi)$ and $g$ being isogenies give rise to isomorphisms $\mathfrak{T}_i(\phi) : T(A)\otimes \Q \xrightarrow{\cong} T(B)\otimes \Q$ and $g :T(A)\otimes \Q \xrightarrow{\cong} T(B) \otimes \Q$. Since we're free to replace $g$ with $n\cdot g$, we may assume that $\mathfrak{T}_i(\phi)^{-1}\circ g \left(T(A)\right) \subseteq T(A).$ However, we recall that as $A$ is assumed to be simple, $T(A)\otimes \Q$ is a dimension 1 vector space over the field $K := \End^0(A),$ and thus $\mathfrak{T}_i(\phi)^{-1}\circ g$ is an element of $\lambda \in K.$ By scaling $g$ further by an integer $n$ we may even assume $\lambda \in \End_{\F_q}(A).$ Thus, $g = \mathfrak{T}_i(\phi)\circ \mathfrak{T}_i(\lambda) = \mathfrak{T}_i(\phi \circ \lambda)$ as desired. 
 \end{proof}
 Proposition \ref{inertclassification} implies that there is a 2-1 map from $\mathcal{C}_h$ to $\mathcal{L}_h$. In particular, given any Deligne module, there exist two almost ordinary abelian varieties corresponding to it, and also two complex Abelian varieties with the same $H_1$ corresponding to the Deligne module. This non-uniqueness is explained by the fact that there are two different CM types on the algebra of endomorphisms associated to $\mathcal{C}_h$, and choosing one of these two different CM types is the same as choosing one of the two different equivalence classes in $\mathcal{C}_h$. 
 \subsection{Ramified isogeny classes}\label{ramifiedclasses}
In the case of ramified isogeny classes, there is no canonical way to associate to $A$ a module $T$, because of the existence of two canonical lifts. However, as we will show, fixing a choice of canonical lift of some one almost ordinary abelian variety fixes a choice of canonical lift for every other abelian variety in the same isogeny class. Indeed, choosing one canonical lift over the other is equivalent to choosing one amongst the two possible CM types on the endomorphism algebra of the ramified isogeny class. 
%Therefore, instead of defining a functor on the class of all ramified almost ordinary abelian varieties, we will instead focus on treating each isogeny class separately. This is enough for applications in the next section. Therefore, We 
To that end, fix a ramified $\F_q$-isogeny class $\mathcal{C}_h$, along with an abelian variety $A \in \mathcal{C}_h$ defined over $\F_q$. We define $\tilde{A}$ to be one of the two canonical lifts of $A$. For the rest of this section, we will call $\tilde{A}$ \emph{the} canonical lift of $A$. 

\begin{proposition}\label{liftsubgroup}
Let $G \subset A$ denote any finite flat subgroup defined over $\F_q$. There is then a canonical subgroup $\tilde{G} \subset \tilde{A}$ lifting $G$. 
\end{proposition}
\begin{proof}
If $G$ is of the form $G_1 \times G_2$, it suffices to prove this result for both $G_1$ and $G_2$. Further, this result is tautologically true for prime-to-$p$ subgroups of $A$. Therefore, we may assume that $G$ has order a power of $p$. Further, we may assume that $G$ is either \'etale, or multiplicative, or local-local. By the definition of the canonical lift, \'etale and multiplicative subgroups lift uniquely so it suffices to prove that every local-local finite flat subgroup of $A$ lifts uniquely to $\tilde{A}$. 

Therefore, let $G \subset A$ be a finite flat subgroup which is local-local. Further, let $A[p^{\infty}] = \G_{et} \times \G_{tor} \times \G_s$. Then, $G \subset \G_s$. Further, we defined the canonical lift of $A$ to correspond to the product of the canonical lifts $\tilde{\G_?}$ of $\G_?$ where $?$ stands for either $et, tor$ or $s$. Therefore, it suffices to prove that $G \subset \G_s$ lifts to a subgroup $\tilde{G} \subset \tilde{\G}_s$. 

It is easy to see that $\G_s$ has a unique order $p^n$ subgroup for each $n$. In fact, as we are dealing with the ramified case, this subgroup equals the $\varpi^n$-torsion of $\G_s$, where $\varpi$ is the uniformizing parameter for $\End_{\F_q}(\G_s)$. Therefore, we may assume that $G = \G_s[\varpi^n]$. As our canonical lift $\tilde{\G}_s$ has the property that every endomorphism of $\G_s$ lifts of $\tilde{\G}_s$, it follows that the action of $\Z_p[\varpi]$ also lifts to $\tilde{\G}_s$. It now follows that the $\varpi^n$ torsion of $\tilde{\G}_s$ is the required lift of $G$. 

\end{proof}

\begin{definition}
\begin{enumerate}
\item Let $G\subset A$ be some finite flat subgroup. We define $\tilde{G} \subset \tilde{A}$ to be the (canonical) lift of $G$ defined in Proposition \ref{liftsubgroup}. 

\item Let $B$ be an abelian variety isogenous to $A$, with $\phi: A \rightarrow B$ an isogeny with kernel $G$. Define $\tilde{B}$ to be the lift of $B$ given by $\tilde{A} / \tilde{G}$.
\end{enumerate}
\end{definition}

\begin{proposition}\label{liftwelldefined}
The lift $\tilde{B}$ is a canonical lift of $B$, and doesn't depend on the choice of the isogeny $\phi$. 
\end{proposition}
\begin{proof}
That the lift is a canonical lift can be checked on the level of $p$-divisible groups. By construction, the $p$-divisible group $\tilde{B}[p^{\infty}]$ is a product of the lifts of the \'etale, multiplicative, and local-local parts of the $p$-divisible group $B[p^{\infty}]$. Further, the action of $\Z_p[\varpi]$ preserves $\tilde{G}$ by construction, and so continues to act on $\tilde{B}[p^{\infty}]$. This proves that $\tilde{B}$ is a canonical lift of $B$. 

In order to prove that this lift is independent of $\phi$ and $G$, suppose that $\phi_1, \phi_2$ are two isogenies between $A$ and $B$, whose kernels are $G_1$ and $G_2$. Define $\tilde{B}_i$ to be the lifts of $B$ which equal $\tilde{A} / \tilde{G_i}$ for $i = 1,2$. By replacing $\phi_2$ with an integer scalar multiple, we may assume that $\phi_2$ factors through $\phi_1$. Therefore, there exists an endomorphism $\alpha \in \End(B)$ such that $\phi_2 = \alpha \circ \phi_1$. But now, the proposition follows from the fact that every endomorphism of $B$ lifts to an endomorphism of $\tilde{B_1}$. 

\end{proof}

\begin{proposition}\label{functoriality}
Every $\F_q$-isogeny $\phi: B \rightarrow C$ lifts uniquely to an isogeny $\tilde{\phi} : \tilde{B} \rightarrow \tilde{C}$ where $\tilde{B},\tilde{C}$ are the lifts provided by Proposition \ref{liftwelldefined}. Thus, the association of $B \mapsto \tilde{B}$ defines a functor from the $\F_q$-isogeny class $\mathcal{C}_h$ of $A$ to the collection of lifts.
\end{proposition}
\begin{proof}
Let $\psi : A \rightarrow B$ be an $\F_q$-isogeny. Let $G := \ker(\psi)$ and $H := \ker(\phi \circ \psi) \supseteq G.$ Clearly, $\tilde{G} \subseteq \tilde{H}$ and moreover the lift $\tilde{\phi}$ corresponds to the natural isogeny $\tilde{B} = \tilde{A}/\tilde{G} \rightarrow \tilde{A}/\tilde{H} = \tilde{C}.$ The fact that this is functorial is also clear.
\end{proof}

Therefore, given an $\F_q$-isogeny $\phi : B \rightarrow C$ (where $B, C$ are abelian varieties in the isogeny class $\mathcal{C}_h$ of $A$) we get a map $\mathfrak{T}(\phi): (T(B),\pi_B) \rightarrow (T(C),\pi_C)$. (Here $T(B)$ refers to $H_1(\tilde{B}\otimes_\epsilon \C,\Z)$ for $\tilde{B}$ being the particular lift defined above.) This defines a functor $\mathfrak{T} : \mathcal{C}_h \rightarrow \mathcal{L}_h$. 

%Moreover, we show that it is fully faithful at least if we restrict to isogenies. 

\begin{proposition}\label{ramifiedclassification}
The functor 
$$\mathfrak{T} : \mathcal{C}_h \rightarrow \mathcal{L}_h$$ is an equivalence of categories.
\end{proposition}
\begin{proof}
By Proposition \ref{surjectivity}, this functor is essentially surjective. That $\mathfrak{T}$ is fully-faithful follows from an argument almost identical to that of Proposition \ref{inertclassification}. 

\end{proof}

%\subsection{Summary}
%We summarize the results of this section below.

%Let $\mathcal{C}$ denote an $\F_q$-isogeny class of \emph{simple} almost ordinary abelian varieties over $\F_q.$ For objects $A, B$ in $\mathcal{C}$ we restrict the set of morphisms between them to be the $\F_q$-isogenies i.e. $\text{Mor}_\mathcal{C}(A,B) := \text{Isog}_{\F_q}(A,B)$. 

%Let $h(x) \in \Z[x]$ denote the characteristic polynomial of the $q^\text{th}$-power Frobenius endomorphisms of varieties in $\mathcal{C}.$ We use the notation introduced earlier in the section. Thus, we have a decomposition in $\Z_p[x],$ $h(x) = h_0(x)\cdot h_1(x) \cdot h_{1/2}(x),$ and $\mathcal{O}$ is the maximal order of $\Q_p[x]/h_{1/2}(x).$ Let $\mathcal{C}'$ denote the isomorphism classes of pairs $(T,F)$ satisfying \ref{Weil},\ref{Ver},$3^*$ and $4^*.$ 

%\begin{theorem}
 %In the case that $\mathcal{O}$ is unramified over $\Z_p$: the equivalence relation defined in \ref{inertclasses} decomposes the isogeny class $\mathcal{C} = \mathcal{C}_1 \cup \mathcal{C}_2,$ and the association $A \mapsto (T(A),\pi_A)$ defines an equivalence of categories $\mathcal{C}_i\xrightarrow{\cong} \mathcal{C}'$ for $=1,2.$
 
 %If $\mathcal{O}$ is ramified over $\Z_p,$ we fix one of two possible choices of canonical lifts for all varieties in $\mathcal{C}$ as was done in \ref{ramifiedclasses}. With this choice the association $A \mapsto (T(A),\pi_A)$ defines an equivalence of categories $\mathcal{C}\xrightarrow{\cong} \mathcal{C}'$
%\end{theorem}

This completes the proof of the classification Theorem \ref{class} stated in the Introduction.

\section{Polarizations}
Throughout this section, we fix an isogeny class $\mathcal{C}_h$ of almost ordinary abelian varieties over $\F_q$ with Frobenius polynomial $h.$ In the case that $\mathcal{C}_h$ is ramified we also fix at once a compatible choice of canonical lifts for all the varieties in $\mathcal{C}_h$ as was done in Section \ref{ramifiedclasses}. Henceforth, we will refer to this choice of canonical lift as \emph{the} canonical lift for any $A \in \mathcal{C}_h$.
%$A \in \mathcal{C}_h$ will denote a simple almost ordinary abelian variety over $\F_q$. Let $K$ denote the field $\End^0(A)$. In the case that $A$ is ramified, we will choose one of its canonical lifts. Note that this fixes a choice of canonical lift for every other $B$ isogenous to $A$.
We note that an $\F_q$-isogeny $\phi: A\rightarrow B$ induces an isomorphism of fields $\End^0(A)\cong \End^0(B).$ After identifying these fields we denote the common endomorphism algebra by $K.$ Finally, as we have fixed an embedding of $W(\overline{\F}_p)[\sqrt{p}]$ in $\C$, every $A \in \mathcal{C}_h$ gives rise to a complex abelian variety with CM by the field $K$, and in the ramified case the compatible choice of lifts amounts to a compatible CM type $\Phi$ of the common endomorphism algebra $K$.

\begin{proposition}\label{polinert}
Let $A$ be an inert almost ordinary abelian variety over $\F_q$, and let $A^\vee$ denote its dual. Then $A^\vee$ is in the same equivalence class as $A$. 
\end{proposition}
\begin{proof}
Whether or not two isogenous abelian varieties are in the same equivalence class doesn't depend on the field of definition. Therefore we may replace $\F_q$ with a finite extension, and assume the existence of a principally polarized $B$ isogenous to $A$. Let $\lambda: B \rightarrow {B^\vee}$ denote a principal polarization. Let $\phi: A \rightarrow B$ denote an isogeny, and let ${\phi^\vee}: {B^\vee} \rightarrow {A^\vee} $ denote its dual. Then, the map $\lambda' = {\phi^\vee} \circ \lambda \circ \phi: A \rightarrow {A^\vee}$ is an isogeny from $A$ to its dual (in fact, a polarization). Further, the finite flat group schemes $\ker(\phi) \cap A[p^{\infty}]_{ss}$ and $\ker({\phi^\vee}) \cap {A^\vee}[p^{\infty}]_{ss}$ have the same cardinality, and so $\ker(\lambda')\cap A[p^{\infty}]_{ss}$ has cardinality a perfect square. Therefore, $A$ and ${A^\vee}$ are in the same equivalence class. 
\end{proof}

We now no longer assume that $A$ is inert. 
\begin{proposition}\label{polarbear}
The canonical lift of ${A^\vee}$ is the dual of the canonical lift of $A$. 
\end{proposition}

\begin{proof}
Let $\tilde{A}$ denote the canonical lift of $A$. Consider ${(\tilde{A})^\vee}$, the dual of $\tilde{A}$ --  its special fiber is the dual of $A$, and hence ${(\tilde{A})^\vee}$ is a lift of ${A^\vee}$. ${(\tilde{A})^\vee}$ has the same endomorphism ring as the canonical lift of $\tilde{A}$, and hence ${A^\vee}$. Therefore it follows that ${\tilde{A}^\vee}$ is the canonical lift of ${A^\vee}$ as required. 

\end{proof}

\begin{definition}
Suppose $(T,F) \in \mathcal{L}_h$ is an almost ordinary Deligne module. Then we define the \emph{dual} module following \cite{Howe} as $(T,F)^{\vee} = (T^{\vee},F^{\vee})$ where $T^{\vee}$ is the $\Z$-module $\Hom_\Z (T,\Z)$ and $F^\vee$ is the endomorphism of $T^\vee$ such that $(F^\vee \psi)(t) = \psi(Vt),$ for all $\psi \in T^\vee$ and $t \in T.$
\end{definition} 
It is easy to see that the dual pair $(T,F)^\vee$ is indeed an almost ordinary Deligne module with the same Frobenius polynomial $h$, and that $\,^\vee: \mathcal{L}_h \rightarrow \mathcal{L}_h$ defines a functor. We also remark that a complex structure on $T\otimes \R$ determines the natural complex structure on $T^\vee \otimes \R = \Hom_\R(T\otimes \R, \R)$ given by $(z\cdot f)(t) = f(\overline{z}\cdot t)$ for $f \in T^\vee \otimes \R$, $t \in T\otimes \R$ and $z\in \C.$ In this manner, the $g$-many complex eigenvalues of $F\otimes \R$ and $F^\vee\otimes \R$ are the same.

\begin{proposition}
Suppose $A^\vee$ denotes the (almost ordinary) abelian variety over $\F_q$ dual to $A$. Then,
$(T(A^\vee),\pi_{A^\vee})$ is the dual of the pair $(T(A),\pi_A)$ defined above.
\end{proposition}
\begin{proof}
The same argument as in \cite[Proposition 4.5]{Howe}.
\end{proof}

Let $\pi \in K$ denote the Frobenius of the isogeny class $\mathcal{C}_h,$ and let $R\subseteq K$ be the smallest order containing $\Z[\pi,q/\pi]$ such that $R \otimes \Z_p$ contains $\mathcal{O}_{ss}$. For an almost ordinary Deligne module $(T,F) \in \mathcal{L}_h$ we may view the isomorphism class of $T$ as an $R$-fractional ideal $I \subset K.$ In fact, it is clear that the isomorphism classes of objects of $\mathcal{L}_h$ are in bijection with the ideal class monoid $\text{ICM}(R)$ of the order $R.$ In the spirit of \cite[Theorem 4.3.2]{marseglia} we rephrase our results in these terms:
\begin{theorem}\label{everything}
For a simple almost ordinary abelian variety $A \in \mathcal{C}_h$, let $I_A \subset K$ denote the associated fractional ideal. Then: 
\begin{enumerate}
\item the dual variety $A^\vee$ corresponds to the fractional ideal $\overline{I}_A^t$ - the CM-conjugate of the trace-dual to $I_A.$
\item $\End_{\F_q}(A)$ is the ring $[I_A:I_A] = \{\lambda \in K : \lambda\cdot I_A \subseteq I_A\}.$
\item \begin{enumerate}
    \item There is a bijection between the $\F_q$-isomorphism classes of varieties in $\mathcal{C}_h$ ($\mathcal{C}_{i,h}$ in the inert case) and the ideal class monoid $\text{ICM}(R)$
    \item $R$ is a Gorenstein order\footnote{By \cite{K}, it suffices to check that $R \otimes \Z_{\ell}$ is Gorenstein for every prime $\ell$. For every $\ell \neq p$, the order is monogenic and hence Gorenstein. For $\ell = p$, it is easy to see that the local order is a direct sum of monogenic rings, and hence is Gorenstein.} and thus there is a bijection between the $\F_q$-isomorphism classes of varieties in $\mathcal{C}_h$ ($\mathcal{C}_{i,h}$ in the inert case) having endomorphism ring exactly $R$ and the ideal class group $\Cl(R).$ 
\end{enumerate}
\item The data of a polarization on $A$ is the same as a $\lambda \in K^\times$ such that: 
\begin{itemize}
    \item the bilinear form on $K$ defined by $(x,y)_{\lambda} := \Trace_{K /\Q}(\lambda x \bar{y})$ is integral on $I_A.$
    \item $\lambda$ is purely imaginary
    \item $\phi(\lambda) / i$ is a positive real number, for $\phi \in \Phi$
\end{itemize}
The polarization corresponding to such a $\lambda$ is principal if and only if the lattice $I_A$ is self-dual for the bilinear form $(\ ,\ )_{\lambda}.$
\item Two pairs $(I,\lambda)$ and $(I',\lambda')$ give rise to isomorphic polarized varieties if and only if there exists $\nu \in K^\times$ such that $I' = \nu I$ and $\lambda = \nu\bar{\nu} \lambda'$.
\end{enumerate}
\end{theorem}

%\begin{proposition}
%Let $I \subset K$ denote the lattice corresponding to $T(A)$. Then, the data of a polarization on $A$ is the same as a bilinear form on $T(A)$ of the form $(x,y)_{\lambda} := \Trace_{K /\Q}(\lambda x \bar{y})$, where: 
%\begin{enumerate}
%\item $\bar{y}$ corresponds to complex conjugation on $K$ applied to $y$.

%\item The element $\lambda$ is an element in $K$, such that $\lambda$ is purely imaginary, and $\phi(\lambda / i)$ is a positive real number, for $\phi \in \Phi$. 

%\item $(\ ,\ )_{\lambda}$ on $I$ is integral. 

%\end{enumerate}
%A pair $(I,\lambda)$ is isomorphic to $(I',\lambda')$ if and only if there exists $\nu \in K$ such that $I' = \nu I$ and $\lambda = \nu\bar{\nu} \lambda'$. Finally, a polarization is self dual if and only if the lattice $T(A)$ is self-dual for the bilinear form $(\ ,\ )_{\lambda}.$ 
%\end{proposition}
\begin{proof}
It only remains to prove the last facts regarding polarizations. We will prove that the set of all polarizations on $\tilde{A}$ is the same as the set of all polarizations on $A$. The proposition follows from this statement. Indeed, a polarization on $\tilde{A}_{\C}$ is the same as a Riemann form on $T(A) = H_1(\tilde{A}_{\C}(\C),\Z)$, and is principal if and only if the associated form is self-dual on $T(A)$. A Riemann form on $T(A)$ is the same data as in the statement of this proposition (see \cite[Example 2.9]{Milne}). 

Therefore, it suffices to prove that the set of all polarizations on $\tilde{A}$ is the same as the set of all polarizations on $A$. By Proposition \ref{polarbear}, $\widetilde{{A^\vee}} = {(\tilde{A})^\vee}$, and we further have that $\Hom(\tilde{A}, \widetilde{{A^\vee}}) = \Hom(A,{A^\vee})$ (this follows from Proposition \ref{functoriality} in the ramified case, and Proposition \ref{polinert} in the inert case). Finally, an element $\alpha \in \Hom(A,{A^\vee})$ is a polarization on $A$ if and only if it is a polarization of $\tilde{A}$.
\end{proof}

\section{Size of isogeny classes}
Let $A$ denote a simple $g$-dimensional almost ordinary abelian variety over $\F_q$, where $g \geq 2$. 
\begin{definition}
Define $I(A,q^n)$ to be the set of principally polarized abelian varieties over $\F_{q^n}$ isogenous to $A$. 
\end{definition}
The goal of this section is to prove the following theorem:
\begin{theorem}\label{lastmain}
Suppose that the Frobenius torus of $A$ has full rank. Then we have the lower bound
$I(A,q^n) \geq q^{n[g(g+1)/2 - 1 + o(1)]/2}$ for a positive density set of $n$.
\end{theorem}
We expect that the Frobenius torus condition is unnecessary, and that the lower bound is actually an equality. Further, this condition is equivalent to the condition in the statement of \cite[Proposition 3.6]{AJ}. We thank Yunqing Tang for pointing this out to us. 

Let $\alpha$ denote a Weil $q$ integer corresponding to the isogeny class containing $A$. Let $R_n$ denote the smallest order inside $\Q(\alpha)$ containing $\alpha^n, q^n/\alpha^n$ and such that $R_n \otimes \Z_p$ contains $\Ok_{ss}$. We have proved that the set of abelian varieties over $\F_{q^n}$ isogenous to $A$ is in bijection with (or admits a 2-1 map onto) the set of equivalence classes of finitely generated $R_n$ submodules of $\Q(\alpha)$. 

Let $R_n^+$ denote the ring $\Z[\alpha^n + q^n/\alpha^n]$. The following proposition is the analogue of Proposition 3.4 in \cite{AJ}:

\begin{proposition}\label{isogclassgroup}
The subset of $I(A,q^n)$ with endomorphism ring exactly equal to $R_n$ is either empty, or admits a bijective\footnote{If the isogeny class is ramified.} / two-one\footnote{If the isogeny class is inert.} map onto the kernel of the norm map 
$$N: \Cl(R_n) \rightarrow \Cl^+(R_n^+).$$
Here, $\Cl^+(R_n^+)$ is the narrow class group of the totally real order $R_n^+$. 
\end{proposition}
The proof follows directly from Theorem \ref{everything}. For more details, see \cite[Proposition 3.5]{AJ}. 

We will also need the analogue of Lemma 3.7 in \cite{AJ}:

\begin{lemma}\label{size}
For a density-one set of positive integers $n$, we have $\frac{\#\Cl(R_n)}{\#\Cl^+(R^+_n)} = (q^{n/2})^{\frac{g(g+1)}{2} -1 + o(1)}$. 
\end{lemma}
\begin{proof}
As $n$ tends to infinity, the class groups of both rings are well approximated by their root-discriminants. 

We first compute the index of $\Z[\alpha^n]$ inside $R_n$. As in \cite{AJ}, this index is a power of $p$, therefore it suffices to compute the corresponding index after tensoring both rings with $\Z_p$. Let $f(x)$ denote the minimal polynomial of $\alpha^n$, and let $f(x) = f_e(x)f_t(x)f_s(x)$ over $\Z_p$ correspond to the slope decomposition of $\G$. As $A$ is almost ordinary, $F_e$ and $f_t$ have degree $g-1$, and $f_s$ has degree two. Let $\beta_1 \hdots \beta_{g-1}$ denote the roots of $f_e$, $\gamma_1 \hdots \gamma_{g-1}$ denote the roots of $f_t$ and $\delta_1,\delta_2$ denote the roots of $f_s$. It follows that the $\beta_i$ are $p$-adic units, that $\frac{v_p(\gamma_i)}{v_p(q^n)} = 1$ and $\frac{v_p(\delta_i)}{v_p(q^n)} = 1/2$. 

Let $g_t$ denote the polynomial with roots $\gamma_i/q$ and $g_s$ denote the polynomial with roots $\delta_i/q^{1/2}.$ The index of $\Z_p[\alpha]$ inside $R_n \otimes \Z_p$ equals the index of $\Z_p[x]/f_t(x)f_s(x)$ inside $\Z_p[x]/g_t(x)g_s(x)$. This index equals the square-root of $\displaystyle \frac{\prod_{ij}(\gamma_i - \gamma_j)^2 \cdot \prod_{ij} (\gamma_i -\delta_j)^2 \cdot (\delta_1-\delta_2)^2}{\prod_{ij}(\gamma_i/q - \gamma_j/q)^2 \cdot \prod_{ij} (\gamma_i/q -\delta_j/q^{1/2})^2 \cdot (\delta_1/q^{1/2}-\delta_2/q^{1/2})^2}$. It is easy to see that the square root of this quotient equals $q^{(g-1)(g-2)/2 + g-1 + 1}$.

Therefore, it suffices to bound the quotient $\frac{d(\Z[\alpha^n])}{d(R^+_n)}$. The same argument as in the last paragraph of \cite[Lemma 3.8]{AJ} goes through verbatim to finish the proof of this result.

\end{proof}

We now prove Theorem \ref{lastmain}. By Lemma \ref{size} and Proposition \ref{isogclassgroup}, it suffices to prove that there is some principally polarized abelian variety with endomorphism ring equal to $R_n$ for a positive density set of $n$. 

\begin{proof}[Proof of Theorem \ref{lastmain}]
Let $h_n$ be the minimal polynomial of $\alpha^n + (q/\alpha)^n $, and let $\lambda_n = [(\alpha^n -(q/\alpha)^n )h'_n(\alpha^n + (q/\alpha)^n)]^-1$. Let $I\subset K$ denote any fractional ideal. The pairing 
\begin{equation}\label{tpair}
(x,y) \mapsto \Trace_{K/\Q}(\lambda_n x\bar{y}) 
\end{equation}
induces a polarization on $I$ precisely when the appropriate positivity conditions on $\lambda_n$ are satisfied. An argument identical to the one in \cite[Proposition 3.6]{AJ} proves that the appropriate positivity conditions hold for a positive proporition of $n$ (namely, a proportion of $\frac{1}{2^g}$). Therefore, we will assume that $n$ is an integer for which $\lambda_n$ satisfies these polarization conditions, and the existence of an abelian variety $A$ with endomorphism ring $R_n$, which is principally polarized. The theorem would follow from this.  We will treat the treat the case when $n = 1$, the same proof goes through verbatim. Let $\mathfrak{b} \subset \mathcal{O}_K$ be the unnique maximal ideal corresponding to the slope-half part of $A$. Let $R = R_1$. As $R$ is locally the maximal order at $\mathfrak{b}$, it follows that the $R$-ideal $\mathfrak{b} \cap R$ is invertible. Further, both $R$ and the ideal $\mathfrak{b}$ are stable under the action of complex conjugation on $K$. 

Let $R' = \Z[\alpha,q/\alpha]$. The dual of $R'$ with respect to the pairing \eqref{tpair} is $R'$ (see \cite[Proposition 9.5]{Howe}). As the orders $R$ and $R'$ agree away from the prime $\mathfrak{b}$, it follows that the dual of $R$ is a power of $\mathfrak{b}$, say $\mathfrak{b}^n$. As $\mathfrak{b}$ is stable under the action of conplex conjugation, the dual of $\mathfrak{b}$ with respect to \eqref{tpair} is $\mathfrak{b}^{n-1}$, the dual of $\mathfrak{b}^2$ is $\mathfrak{b}^{n-2}$, etc. If $n$ was an even integer, then the ideal $\mathfrak{b}^{n/2}$ is self dual, thereby yielding a principally polarized abelian variety, as required. We will now prove that if the isogeny class is ramified, then $n$ necessarily has to be even, and if the isogeny class is inert, we will produce an abelian variety which is principally polarized. 
\subsection*{The ramified case}
Suppose that $n$ were odd. Without loss of generality, we assume that $n = 1$. Therefore, there exists a polarization with degree equal  to the size of $R / \mathfrak{b}$. However, $R / \mathfrak{b}$ has size $p$, and the degree of a polarization is necessarily a square, yielding a contradiction. Therefore, $n$ had to have been even, yielding the required result in the ramified case. 
\subsection*{The inert case}
Again, we assume that $n = 1$. Let $A$ denote an abelian variety (in either equivalence class) corresponding to the ideal $R$. The polarization constructed is of the form $\lambda: A \rightarrow \check{A}$, and has kernel equal to the $p$-torsion of the supersingular part of $A[p^{\infty}]$. Let $B$ be the abelian variety such that $B / G = A$, where $G \subset B[p]$ is the $p$-torsion of the \'etale part of $B[p^{\infty}]$. Then, the dual isogeny from $\check{A}$ to $\check{B}$ has kernel equal to the $p$-torsion of the multiplicative part of $\check{A}[p^{\infty}]$. Therefore, the composite map from $B$ to $\check{B}$ has kernel $B[p]$, and thus $\check{B} = B / B[p] = B$. We have produced a principally polarized abelian variety $B$, isogenous to $A$! It is also clear that the endomorphism ring of $B$ equals that of $A$, whence the theorem follows.

\end{proof}

\bibliographystyle{alpha}

\begin{thebibliography}{9}

\bibitem{CJ}
T.G. Centeleghe, J. Stix. Categories of abelian varieties over finite fields, I: Abelian varieties over $\F_p$. 
{\it Algebra Number Theory} 9 (2015), no. 1, 225–265. 

\bibitem{Chi}
W.C. Chi. $\ell$-adic and $\lambda$-adic representations associated to abelian varieties defined over number fields. 
{\it Amer. J. Math.} 114 (1992), no. 2, 315–353. 

\bibitem{Deligne}
P. Deligne. Variétés abéliennes ordinaires sur un corps fini.  {\it Invent. Math.} 8 1969 238–243.

\bibitem{Deuring}
M. Deuring. Die Typen der Multiplikatorenringe elliptischer Funktionenkörper. {\it Abh. Math. Sem. Univ. Hamburg} 14 (1941), no. 1, 197–272.

\bibitem{Drin}
V.G. Drinfel'd. Coverings of p-adic symmetric domains. {\it Funkcional. Anal. i Priložen.} 10 (1976), no. 2, 29–40.  

\bibitem{Howe}
E. Howe. Principally polarized ordinary abelian varieties over finite fields.
{\it Trans. Amer. Math. Soc.} 347 (1995), no. 7, 2361–2401. 

\bibitem{K}
S.L. Kleiman, J.O. Kleppe. Macaulay duality over any base. {\it Preliminary preprint dated 26th December 2018.}

\bibitem{marseglia}
S. Marseglia. Isomorphism classes of abelian varieties over finite fields, {\it PhD Thesis}, 2016.

\bibitem{GM}
W. Messing. The crystals associated to Barsotti-Tate groups: with applications to abelian schemes. Lecture Notes in Mathematics, Vol. 264. {\it Springer-Verlag}, Berlin-New York, 1972.

\bibitem{Milne}
J.S. Milne. Complex Multiplication. www.jmilne.org/math/CourseNotes/CM.pdf 

\bibitem{AJ}
A.N. Shankar, J. Tsimerman. (2018). Unlikely Intersections in finite characteristic. {\it Forum of Mathematics, Sigma}, 6, E13. 

\bibitem{Waterhouse}
W.C. Waterhouse. Abelian varieties over finite fields. {\it Ann. Sci. École Norm. Sup.} (4) 2 1969 521–560.

\end{thebibliography}

\end{document}